\documentclass[a4paper]{amsart}

\usepackage[latin1]{inputenc}
\usepackage{amsfonts}
\usepackage{amsthm}
\usepackage{amssymb}
\usepackage[T1]{fontenc}
\usepackage{mathtools}
\usepackage{enumitem}
\usepackage{tikz-cd}
\usepackage{tikz}
\usetikzlibrary{arrows}
\usetikzlibrary{matrix}
\usepackage{mathrsfs}

\renewcommand{\leq}{\leqslant}
\renewcommand{\geq}{\geqslant}

\usepackage{inputenc}
\inputencoding{utf8}
\makeatother

\theoremstyle{plain}
\newtheorem{thm}{Theorem}[section]
\newtheorem*{thm*}{Theorem}
\newtheorem{prop}[thm]{Proposition}
\newtheorem*{prop*}{Proposition}

\newtheorem{lem}[thm]{Lemma}

\newtheorem{cor}[thm]{Corollary}

\theoremstyle{definition}
\newtheorem{defn}[thm]{Definition}

\theoremstyle{remark}
\newtheorem{rem}[thm]{Remark}

\newtheorem{example}[thm]{Example}

\newcommand{\R}{\mathbb{R}}

\newcommand{\G}{\sigma\ell\mathbb{G}}
\newcommand{\Gu}{\sigma\ell\mathbb{G}_u}
\newcommand{\RS}{\sigma\mathbb{RS}}
\newcommand{\RSu}{\sigma\mathbb{RS}_u}

\newcommand{\VG}{\mathcal{V}_{\sigma\ell\mathbb{G}}}
\newcommand{\VGu}{\mathcal{V}_{\sigma\ell\mathbb{G}_u}}
\newcommand{\VRS}{\mathcal{V}_{\sigma\mathbb{RS}}}
\newcommand{\VRSu}{\mathcal{V}_{\sigma\mathbb{RS}_u}}

\title{Dedekind $\sigma$-complete $\ell$-groups and Riesz spaces as varieties}
\author[Marco Abbadini]{Marco Abbadini}
\email{marco.abbadini@unimi.it}
\address{Dipartimento di Matematica {\sl Federigo Enriques}, Universit\`a degli Studi di Milano, via Cesare Saldini 50, 20133 Milano, Italy.}


\begin{document}
\maketitle
\begin{abstract}We prove that the category of Dedekind $\sigma$-complete Riesz spaces is an infinitary variety, and we provide an explicit equational axiomatization. In fact, we show that finitely many axioms suffice over the usual equational axiomatization of Riesz spaces. Our main result is that  $\R$, regarded as a Dedekind $\sigma$-complete Riesz space, generates this category as a quasi-variety, and therefore as a variety. Analogous results are established for  the categories of (i) Dedekind $\sigma$-complete Riesz spaces with a  weak order unit, (ii)  Dedekind $\sigma$-complete lattice-ordered groups, and (iii) Dedekind $\sigma$-complete lattice-ordered groups with a  weak order unit.
\end{abstract}

\section{Introduction}\label{S. intro}
Riesz spaces --- also known as vector lattices --- and the more general lattice-ordered groups are of importance both in analysis (see e.g.\ \cite{Aliprantis_Border2006, Aliprantis_Burkinshaw}) and in algebra (see e.g.\ \cite{BKW}). In connections with integration and measure theory the conditionally countably complete Riesz spaces, known as Dedekind $\sigma$-complete Riesz spaces, are particularly relevant. We recall that a Riesz space $G$ is \emph{Dedekind $\sigma$-complete} if for all countable subsets $S\subseteq G$, if $S$ admits an upper bound in $G$, then $S$  admits a least upper bound. Let us write $\RS$ for the category whose objects are such Riesz spaces and whose morphisms are the Riesz morphisms (=vector space and lattice homomorphisms) that preserve existing countable suprema. While Riesz spaces and their morphisms  form a variety of algebras in the sense of Birkhoff (see e.g.\ \cite{Burris_Sanka}), to our knowledge no equational presentation of the category $\RS$ is available in the literature. Our first result is an explicit axiomatization of $\RS$ as an infinitary variety  (Theorem \ref{t. T and E are inverse functors}). Specifically, $\RS$ can be presented as an equationally definable class of algebraic structures with the same primitive operations as Riesz spaces, and one additional operation of countably infinite arity which we write as $\bigvee\limits^-$. Semantically, $\bigvee\limits^-(g,f_1,f_2,\dots)$ is interpreted in a Dedekind $\sigma$-complete Riesz space as $\sup_{n\geq 1 }\{f_n\land g \}$. We prove that, in addition to the usual Riesz space equational axioms, finitely many equations suffice to axiomatize  $\RS$  in this language, and these express very basic properties of $\sup_{n\geq 1  }\{f_n\land g \}$. 
The set $\R$ is a fundamental example both of Riesz space and of  Dedekind $\sigma$-complete Riesz space. It is well known that $\R$ generates the (finitary) variety of Riesz spaces, see \cite[Chapter XV]{Birkhoff}. Our main result shows that, analogously, $\R$ generates the infinitary variety of  Dedekind $\sigma$-complete Riesz spaces (Theorem \ref{t. generation r}). In fact, in Corollary \ref{c. generation as quasi-variety} we obtain the stronger result that  $\R$ generates this variety as a quasi-variety, too.

Additionally, for each of the results mentioned above we prove an analogous counterpart for the category $\RSu$ whose objects are  Dedekind $\sigma$-complete Riesz spaces with a designated weak (order) unit, where  the morphisms are the Riesz morphisms that preserve both existing countable suprema and  the weak unit. Finally, we obtain analogous results for corresponding categories of lattice-ordered groups, henceforth shortened to \emph{$\ell$-groups}. For this, we will consider the category $\G$ whose objects are Dedekind $\sigma$-complete $\ell$-groups, and the category  $\Gu$ whose objects are Dedekind $\sigma$-complete $\ell$-groups with a designated weak unit.

We assume familiarity with the basic theory of $\ell$-groups and Riesz spaces. All needed background can be found, for example, in the standard references \cite{BKW, LuxZaan}.

\subsection*{Acknowledgements} The author would like to thank prof. Gerard J.H.M. Buskes for his help with the literature related to Theorem \ref{t. Nepalese}. Some of the results in this paper were originally obtained by the author as part of his M.Sc.\ Thesis, written under the supervision of prof. V. Marra at the University of Milan, Italy.

\section{Dedekind $\sigma$-complete $\ell$-groups are a variety}\label{S. l-gr}

\subsection{Definition of $\G$ and $\VG$}\label{s. V is sigma}
	\begin{defn}
		A morphism of $\ell$-groups (or $\ell$-morphism) $\varphi\colon G\to H$ is said to be \emph{$\sigma$-continuous} if $\varphi$ preserves the existing countable suprema.
	\end{defn}
	\begin{defn}
		We denote by $\G$ the category whose objects are Dedekind $\sigma$-complete $\ell$-groups and whose arrows are $\sigma$-continuous $\ell$-morphisms.
	\end{defn}
	
	It is well-known that a Dedekind $\sigma$-complete $\ell$-group is archimedean (if $a\geq 0$ and $na\leq b$ for every $n\geq 1$, then $a=0$) and thus abelian.
	
	\begin{defn}\label{d. V}
		Let $\VG$ be the (infinitary) variety described in the following.
		\begin{description}
			\item[{\it Operations of $\VG$}] operations of $\ell$-groups and an operation $\bigvee\limits^-$ of countably infinite arity. The intended interpretation of $\bigvee\limits^-(g,f_1,f_2,\dots)$ in a Dedekind $\sigma$-complete $\ell$-group is $\sup_{n\geq 1  }\{f_n\land g \}$, and we adopt the notation
			$$\bigvee\limits_{n\geq 1  }^g f_n\coloneqq \bigvee\limits^-\left(g,f_1,f_2,\dots\right)$$
			\item[{\it Axioms of $\VG$}] The axioms of $\VG$ are the axioms of $\ell$-groups and the following ones
			(which are seen to be equational, once we rewrite every inequality $a\leq b$ as $a\land b=a$)
			\begin{enumerate}
				\item[(A1)] $\bigvee\limits_{n \geq 1  }^g f_n=\bigvee\limits_{n \geq 1  }^g (f_n\land g)$;
				\item[(A2)]  $\bigvee\limits_{n\geq 1 }^g f_n=(f_1\land g)\lor\left(\bigvee\limits_{n\geq 2 }^g f_n\right)$;\\
				\item[(A3)] $\bigvee\limits_{n \geq 1  }^g (f_n\land h)\leq h$.
			\end{enumerate}
		\end{description}
	\end{defn}

	\begin{lem}\label{l. A2 more natural}
		In $G\in\VG$, for every $k\geq 1$, we have
		$$f_k\land g\leq \bigvee\limits_{n \geq 1  }^g f_n.$$
	\end{lem}
	\begin{proof}
		By induction on $k\geq 1$ and applying (A2)   we obtain
			$$\bigvee\limits_{n\geq 1 }^g f_n=(f_1\land g)\lor\dots\lor (f_k\land g) \lor\left(\bigvee\limits_{n\geq k+1 }^g f_{n}\right).$$
		Thus $f_k\land g\leq (f_1\land g)\lor\dots\lor (f_k\land g) \lor\left(\bigvee\limits_{n\geq k+ 1 }^g f_{n}\right) = \bigvee\limits_{n\geq 1 }^g f_n$.
	\end{proof}
	
\subsection{The categories $\G$ and $\VG$ are isomorphic}

	\begin{prop}\label{supinf}
		Let $G\in \VG$. In $G$ we have
		$$\bigvee\limits_{n \geq 1  }^g f_n=\sup_{n \geq 1  }\{f_n\land g\}.$$
	\end{prop}
	\begin{proof}
				By Lemma \ref{l. A2 more natural}, $\bigvee\limits_{n \geq 1  }^g f_n$ is an upper bound of $(f_k\land g)_{k\geq 1  }$. Suppose now that $f_n\land g\leq h$ for every $n\geq 1  $. Then
				$$\bigvee_{n\geq 1  }^g f_n\stackrel{\text{(A1)}}{=}\bigvee_{n\geq 1  }^g (f_n\land g)\stackrel{ f_n\land g\leq h}{=}\bigvee_{n\geq 1  }^g (f_n\land g\land h)\stackrel{\text{(A3)}}{\leq}h.$$
	\end{proof}

\begin{cor}\label{c. Dedekind sigma}
	Let $G\in\VG$. Then $G$ is a Dedekind $\sigma$-complete $\ell$-group.
\end{cor}
\begin{proof}
	Let $(f_n)_{n\geq 1  }\subseteq G$ and $g\in G$ be such that $f_n\leq g$ for all $n\geq 1  $. Then 
	$$\bigvee\limits_{n \geq 1  }^g f_n\stackrel{\text{Prop. }\ref{supinf}}{=}\sup_{n \geq 1  }\{f_n\land g\}\stackrel{f_n\leq g}{=}\sup_{n \geq 1  }f_n.$$
\end{proof}

\begin{lem}\label{l. sups are preserved}
	Let $\varphi\colon G\to H$ be a morphism in $\VG$. Then $\varphi$ is $\sigma$-continuous.
\end{lem}
	\begin{proof}
		 Let $(f_n)_{n\geq 1  }\subseteq G$ and $f=\sup_{n \geq 1  } f_n$. Then
		$$\varphi\left(\sup_{n\geq 1  }f_n	\right)\stackrel{f_n\leq f}{=}\varphi\left(\sup_{n\geq 1  }\{f_n\land f\}	\right)\stackrel{\text{Prop. }\ref{supinf}}{=}\varphi\left(\bigvee\limits_{n \geq 1  }^f f_n\right)\stackrel{\varphi \text{ preserves }\bigvee\limits^-}{=}\bigvee\limits_{n \geq 1  }^{\varphi(f)} \varphi(f_n)=$$
		$$\stackrel{\text{Prop. }\ref{supinf}}{=}\sup_{n\geq 1  }\{\varphi(f_n)\land \varphi(f)\}\stackrel{\varphi \text{ preserves }\land}{=}\sup_{n\geq 1  }\varphi(f_n\land f)\stackrel{f_n\leq f}{=}\sup_{n\geq 1  }\varphi(f_n).$$
	\end{proof}

	We denote by $U$ the forgetful functor
	$$U\colon \VG\to \G$$
	that assigns to an object $G\in \VG$ the set $G$, endowed with the operations of $\ell$-group of $G$ (we forget the operation $\bigvee\limits^-$). For $\varphi\colon G\to H$ morphism in $\VG$, we set $U(\varphi)\coloneqq \varphi$.
	
	\begin{prop}\label{r. U is well-defined}
		The functor $U$ is well-defined.
	\end{prop} 
	\begin{proof}
		Every $G\in \VG$ is a Dedekind $\sigma$-complete $\ell$-group by Corollary \ref{c. Dedekind sigma}.
		
		Moreover, every $\varphi\colon G\to H$ morphism in $\VG$ is an $\ell$-morphism (because $\varphi$ preserves the operations of $\VG$) which, by Lemma \ref{l. sups are preserved}, is $\sigma$-continuous.
	\end{proof}
		We denote by $F$ the functor 
		$$F\colon \G\to \VG$$
		that assigns to an object $G\in  \G$ the set $G$, endowed with the operations of $\ell$-group of $G$, and enriched with the operation $\bigvee\limits_{n \geq 1  }^g f_n\coloneqq \sup_{n\geq 1  }\{f_n\land g\}$. Such supremum exists because $G$ is Dedekind $\sigma$-complete and the countable family $(f_n\land g)_{n\geq 1  }$ is bounded from above by $g$. For $\varphi\colon G\to H$ morphism in $\G$, we set $F(\varphi)\coloneqq \varphi$.

	\begin{prop}\label{r. F well-defined}
		The functor $F$ is well-defined.
	\end{prop}
	\begin{proof}
			Let $G\in\G$. The axioms of $\ell$-groups are obviously satisfied by $F(G)$. The structure of $\ell$-group (and hence the order) is preserved by the action of $F$ on objects. Moreover, (A1),  (A2) and  (A3) in Definition \ref{d. V} hold in $F(G)$ because in $G$ we have
			\begin{enumerate}	
				\item $\sup_{n \geq 1  }(f_n\land g)=\sup_{n \geq 1  } ((f_n\land g)\land g)$;
				\item  $\sup_{n \geq 1  } (f_n\land g)=(f_1\land g)\lor \sup_{n \geq 2  } (f_n\land g)$;		
				\item $\sup_{n \geq 1  } (f_n\land h\land g)\leq h$.
			\end{enumerate}
		
			Let $\varphi\colon G\to H$ be a morphism in $\G$. $\varphi$ preserves the operations of $\ell$-group. Let $(f_n)_{n\geq 1  }\subseteq G$ and $g\in G$.  Then
			$$\varphi\left(\bigvee_{n\geq 1  }^g f_n\right)\stackrel{\text{def. of }F}{=}\varphi\left(\sup_{n\geq 1  } \{f_n\land g\}\right)\stackrel{\text{$\varphi$ preserves countable sups}}{=}\sup_{n\geq 1  }\varphi(f_n\land g)=$$
			$$\stackrel{\text{$\varphi$ preserves $\land$}}{=}\sup_{n\geq 1  }\{\varphi(f_n)\land \varphi(g)\}\stackrel{\text{def. of }F}{=}\bigvee_{n\geq 1  }^{\varphi(g)} \varphi(f_n).$$		
			Hence, $\varphi$ preserves the operation $\bigvee\limits^-$.
	\end{proof}

	\begin{thm}\label{t. U and F are inverse functors}
		$U\colon \VG\to \G$ and $F\colon \G\to \VG$ are inverse functors.
	\end{thm}
	\begin{proof} 
		Let $G\in\G$. Then $UF(G)=G$ since both the functor $F$ and the functor $U$ preserve the operations of $\ell$-groups. Moreover, for $\varphi$ a morphism in $\G$, $UF(\varphi)=U(\varphi)=\varphi$.

	Let $G\in \VG$. Then $FU(G)=G$ because both the functor $F$ and the functor $U$ preserve the operations of $\ell$-groups, and the element $\bigvee\limits^g_{n\geq 1  } f_n$ in $FU(G)$ is, by definition of $F$, the element $\sup_{n\geq 1  } \{f_n\land g\}$ in $U(G)$, which, by Proposition \ref{supinf}, is the element  $\bigvee\limits^g_{n\geq 1  } f_n$ in $G$.	Moreover, for $\varphi$ morphism in $\VG$, $FU(\varphi)=F(\varphi)=\varphi$.
	\end{proof}

\begin{cor}
	The category of Dedekind $\sigma$-complete  $\ell$-groups is an infinitary variety.
\end{cor}

\section{Dedekind $\sigma$-complete  $\ell$-groups with weak  unit are a variety} \label{S. l-gr 1}
\begin{defn}\label{d. weak unit}
	An element $1$ of an  $\ell$-group $G$ is a \emph{weak \textup{(}order\textup{)} unit} iff $1\geq0$ and, for all $f\in G$, 
	$$f\land 1 =0\Rightarrow f=0.$$
\end{defn}

\subsection{Definition of $\Gu$ and $\VGu$}
		\begin{defn}
		We denote by $\Gu$ the category whose objects are Dedekind $\sigma$-complete  $\ell$-groups with a designated weak  unit $1$ and whose morphisms are $\sigma$-continuous $\ell$-morphisms which preserve such designated weak unit.
	\end{defn}
	
	\begin{defn}\label{d. V_u}
		Let $\VGu$ be the (infinitary) variety described in the following.
		\begin{description}
			\item[{\it Operations of $\VGu$}] operations of $\VG$ (see Definition \ref{d. V}) and the constant symbol $1$.
			\item[{\it Axioms of $\VGu$}]
				The axioms of $\VGu$ are the axioms of $\VG$ (see Definition \ref{d. V}) and
				\begin{align}\label{axiom for 1}
					\bigvee\limits_{n\geq1}^{|f|}(|f|\land n1)=|f|.
				\end{align}
		\end{description}	
	\end{defn}

\subsection{Every $G\in \VGu$ is a Dedekind $\sigma$-complete  $\ell$-group with weak  unit}

Every $G\in\VGu$ is a Dedekind $\sigma$-complete  $\ell$-group in an obvious way, as shown in Section \ref{s. V is sigma}. This section shows that $1$ is a weak  unit for $G$.

\begin{lem}\label{l. somma di disgiunti e disgiunta}
	Let $G$ be an abelian $\ell$-group. Let $a,b,c\in G$. If $a\land c=0$ and $b\land c=0$, then $(a+b)\land c=0$.
\end{lem}
\begin{proof}
	Let us suppose $a\land c=0$ and $b\land c=0$. As a first consequence, $a,b,c\geq0$. $0=0+0=a\land c+b\land c\stackrel{+\text{ distributes over }\land}{=}(a+(b\land c))\land (c+(b\land c))\stackrel{+\text{ distributes over }\land}{=}(a+b)\land(\underbrace{a+c}_{\geq c})\land (\underbrace{c+b}_{\geq c})\land (\underbrace{c+c}_{\geq c})\geq (a+b)\land c\land c\land c=(\underbrace{a+b}_{\geq 0})\land \underbrace{c}_{\geq 0}\geq 0$. Thus $(a+b)\land c=0$.
\end{proof}
\begin{lem}\label{l. minus distributes}
	Let $G$ be an abelian $\ell$-group. Let $a\in G$. Then $(na)^-=na^-$.
\end{lem}
\begin{proof}
	$-na^-=0\land (na)=(0+na^+ -na^-)\land (na^+-na^-)\stackrel{+\text{ distributes over }\land}{=}((0+na^-)\land (na^+)-na^-=(na^-\land na^+)-na^-\stackrel{\text{Lemma }\ref{l. somma di disgiunti e disgiunta}}{=}-na^-$.
\end{proof}

\begin{prop}\label{p. 1 is weak}
	Let $G\in\VGu$. The element $1$ of $G$ is a weak  unit.
\end{prop}
\begin{proof}
	\begin{enumerate}
		\item {\it Claim: $1\geq 0$}. 
		
		By Item \eqref{axiom for 1} in Definition \ref{d. V_u}, taking $f=0$, we have
		$\bigvee\limits_{n\geq 1}^{\lvert 0\rvert}(\lvert 0\rvert \land n1)=\lvert 0\rvert$. Thus $0=\sup_{n\geq 1}\{0\land n1 \}=\sup_{n\geq 1}-(n1)^-\stackrel{\text{Lemma \ref{l. minus distributes}}}{=}\sup_{n\geq 1}-n1^-=-1^-$. Thus $1^-=0$, and $1=1^+-1^-=1^+\geq 0$.

		\item {\it Claim: $f \land 1=0\Rightarrow f=0$.}
		
		Suppose $f\land 1=0$. As a first consequence, $f\geq0$. By induction on $n$, Lemma \ref{l. somma di disgiunti e disgiunta} proves $f \land n1=0$ for every $n\geq 1$. Thus,
		$$f\stackrel{f\geq 0}{=}\lvert f\rvert\stackrel{\text{Ax.\ }\eqref{axiom for 1}\text{ in Def.\ }\ref{d. V_u}}{=}\bigvee\limits_{n\geq1}^{|f|}(|f|\land n1)=\bigvee\limits_{n\geq1}^{f}(f\land n1)=\bigvee\limits_{n\geq1}^{f}0=0.$$
	\end{enumerate}
\end{proof}

\subsection{The categories $\Gu$ and $\VGu$ are isomorphic}
We denote by $U_u$ the forgetful functor
$$U_u\colon \VGu\to \Gu$$
that assigns to an object $G\in \VGu$ the set $G$, endowed with the operations of $\ell$-group of $G$, and with the interpretation in $G$ of the constant symbol $1$ as the designated weak  unit (we forget the operation $\bigvee\limits^-$). For $\varphi\colon G\to H$ morphism in $\VGu$, we set $U_u(\varphi)\coloneqq \varphi$.

\begin{prop}
	The functor $U_u$ is well-defined.
\end{prop} 
\begin{proof}
	Let $G\in \VGu$. Then $G$ is a Dedekind-$\sigma$-complete $\ell$-group by Corollary \ref{c. Dedekind sigma}, and the element $1$ of $G$ is a weak unit by Proposition \ref{p. 1 is weak}.
	
	Let $\varphi\colon G\to H$ be a morphism in $\VGu$. $\varphi$ is a $\sigma$-continuous $\ell$-morphism, as shown in Proposition \ref{r. U is well-defined}. Moreover, $\varphi$ preserves the element $1$, by definition of morphism in $\VGu$.
\end{proof}

	\begin{lem}\label{l. distriutivity a la BKW}
	Let $G$ be an $\ell$-group, $I\neq\emptyset$ a set and $(x_i)_{i\in I}\subseteq G$. If $\sup_{i\in I}x_i$ exists, then, for every $a\in G$, $\sup_{i\in I} \{a\land x_i \}$ exists and
	$$a\land \left(\sup_{i\in I} x_i\right)=\sup_{i\in I}\{a\land x_i \}.$$
\end{lem}
\begin{proof}
	See Proposition 6.1.2 in \cite{BKW}.
\end{proof}

\begin{prop}\label{p. axiom for 1}
	Let $G\in\Gu$, $f\in G^+$. Then
	$$f=\sup_{n \geq 1}\{f\land n1 \}.$$
\end{prop}
\begin{proof}
	Set $f_n\coloneqq f\land n1$ and $g\coloneqq \sup_{n \geq 1} f_n$.	
	$$f_{n+1}=f\land (n+1)1=f\land(n1+1)\stackrel{f+1\geq f}{=}[(f+1)\land(n1+1)]\land f=$$
	$$=((f\land n1) +1)\land f=(f_n+1)\land f.$$	
	$$g=\sup_{n\geq 1}f_n\stackrel{f_1\leq f_2}{=}\sup_{n \geq 2}f_n=\sup_{n \geq 1}f_{n+1}=\sup_{n \geq 1}\{(f_n+1)\land f\}=$$
	
	$$\stackrel{\text{Lemma }\ref{l. distriutivity a la BKW}}{=}\left(\sup_{n \geq 1}\{f_n+1\}\right)\land f=\left(\left(\sup_{n \geq 1}f_n\right)+1\right)\land f=(g+1)\land f.$$
	
	From $g=(g+1)\land f$ we obtain $0=((g+1)\land f)-g$. By distributivity of $+$ over $\land$, we obtain $0=(g+1-g)\land (f-g)=1\land (f-g)$. Being $1$ a weak unit, we obtain $f-g=0$, hence the thesis.
\end{proof}

We denote by $F_u$ the functor 
$$F_u\colon \Gu\to \VGu$$
that assigns to an object $G\in  \Gu$ the set $G$, endowed with the operations of  $\ell$-group of $G$, and enriched with the operation $\bigvee\limits_{n \geq 1  }^g f_n\coloneqq \sup_{n\geq 1  }\{f_n\land g\}$ and with the designated weak  unit of $G$ as the interpretation of the constant symbol $1$. For $\varphi\colon G\to H$ be a morphism in $\Gu$, we set $U_u(\varphi)\coloneqq \varphi$.

\begin{prop}\label{r. F_u is well-defined}
	The functor $F_u$ is well-defined.
\end{prop}
\begin{proof}
Let $G\in \Gu$. Then $F_u(G)$ satisfies the axioms of $\VG$, by Proposition \ref{r. F well-defined}.
Moreover, the axiom $\bigvee\limits_{n\geq1}^{|f|}(|f|\land n1)=|f|$ holds in $F_u(G)$, by Proposition \ref{p. axiom for 1}.

	Let $\varphi\colon G\to H$ be a morphism in $\Gu$. $\varphi$ preserves the operations of $\VG$, as shown in Proposition \ref{r. F well-defined}. Moreover, $\varphi$ preserves the weak  unit $1$, by definition of morphism in $\sigma\ell \mathbb{G}_u$.
\end{proof}

\begin{thm}
	$U_u\colon \VGu\to \Gu$ and $F_u\colon \Gu\to \VGu$ are inverse functors.
\end{thm}
\begin{proof}
	Let $G\in\Gu$. Then $U_uF_u(G)=G$  since both the functor $F_u$ and the functor $U_u$ preserve the operations of  $\ell$-groups and the designated unit. Moreover, for $\varphi$ morphism in $\Gu$, $U_uF_u(\varphi)=U_u(\varphi)=\varphi$.

	Let $G\in \VGu$. Then $F_uU_u(G)=G$ because the operations of $\VG$ over $G$ are preserved by $F_u U_u$, as shown in Theorem \ref{t. U and F are inverse functors}, and $1_{F_u U_u (G)}=1_{U_u (G)}=1_G$. Moreover, for $\varphi$ morphism in $\VGu$, $F_uU_u(\varphi)=F_u(\varphi)=\varphi$.
\end{proof}

\begin{cor}
	The category of Dedekind $\sigma$-complete  $\ell$-groups with weak unit is an infinitary variety.
\end{cor}

\section{Dedekind $\sigma$-complete Riesz spaces are a variety}
In this section we will omit the proofs, since the assertions can be reduced to (or done in analogy with) the results in Section \ref{S. l-gr}.
\subsection{Definition of $\RS$ and $\VRS$}

\begin{defn}
	We denote by $\RS$ the category whose objects are Dedekind $\sigma$-complete Riesz spaces and whose morphisms are morphisms of Riesz spaces (or \emph{Riesz morphisms}) which are $\sigma$-continuous.
\end{defn}

\begin{defn}\label{d. W}
	Let $\VRS$ be the (infinitary) variety described in the following.
	\begin{description}
		\item[{\it Operations of $\VRS$}] operations of Riesz spaces and an operation of countably infinite arity $\bigvee\limits^-$.
		\item[{\it Axioms of $\VRS$}]The axioms of $\VRS$ are the axioms of Riesz spaces and the following ones (which are seen to be equational, once we rewrite every inequality $a\leq b$ as $a\land b=a$)
		\begin{enumerate}
			\item[(A1)] $\bigvee\limits_{n \geq 1  }^g f_n=\bigvee\limits_{n \geq 1  }^g (f_n\land g)$;
			\item[(A2)]  $\bigvee\limits_{n\geq 1 }^g f_n=(f_1\land g)\lor\left(\bigvee\limits_{n\geq 2 }^g f_n\right)$;\\
			\item[(A3)] $\bigvee\limits_{n \geq 1  }^g (f_n\land h)\leq h$.
		\end{enumerate}
	\end{description}
\end{defn}

\subsection{The categories $\RS$ and $\VRS$ are isomorphic}

We denote by $T$ the forgetful functor
$$T\colon \VRS\to \RS$$
that assigns to an object $G\in \VRS$ the set $G$, endowed with the operations of Riesz space of $G$ (we forget the operation $\bigvee\limits^-$). For $\varphi\colon G\to H$ morphism in $\VRS$, we set $T(\varphi)\coloneqq \varphi$.

We denote by $E$ the functor 
$$E\colon \RS\to \VRS$$
that assigns to an object $G\in  \RS$ the set $G$, endowed with the operations of Riesz space of $G$, and enriched with the operation $\bigvee\limits_{n \geq 1  }^g f_n\coloneqq \sup_{n\geq 1  }\{f_n\land g\}$. For $\varphi\colon G\to H$ morphism in $\RS$, we set $E(\varphi)\coloneqq \varphi$.

\begin{thm}\label{t. T and E are inverse functors}
	$T\colon \VRS\to \RS$ and $E\colon \RS\to \VRS$ are inverse functors.
\end{thm}

\begin{cor}\label{c.RSvar}
	The category of Dedekind $\sigma$-complete Riesz spaces is an infinitary variety.
\end{cor}

\section{Dedekind $\sigma$-complete Riesz spaces with weak  unit are a variety}
In this section we will omit the proofs, since the assertions can be reduced to (or done in analogy with) the results in Section \ref{S. l-gr 1}.
\subsection{Definition of $\RSu$ and $\VRSu$}
\begin{defn}
	We denote by $\RSu$ the category whose objects are Dedekind $\sigma$-complete Riesz spaces with a designated weak unit $1$, and whose morphisms are $\sigma$-continuous Riesz morphisms which preserve such designated weak unit.
\end{defn}

\begin{defn}\label{d. W_u}
	Let $\VRSu$ be the (infinitary) variety described in the following.
	\begin{description}
		\item[{\it Operations of $\VRSu$}] operations of $\VRS$ (see Definition \ref{d. W}) and the constant symbol $1$.
		\item[{\it Axioms of $\VRSu$}]
			The axioms of $\VRSu$ are the axioms of $\VRS$ (see Definition \ref{d. W}) and $$\bigvee\limits_{n\geq1}^{|f|}(|f|\land n1)=|f|.$$
	\end{description}
\end{defn}

\subsection{The categories $\RSu$ and $\VRSu$ are isomorphic}
We denote by $T_u$ the forgetful functor
$$T_u\colon \VRSu\to \RSu$$
that assigns to an object $G\in \VRSu$ the set $G$, endowed with the operations of Riesz space of $G$ and with the interpretation in $G$ of the constant symbol $1$ as the designated weak  unit (we forget the operation $\bigvee\limits^-$). For $\varphi\colon G\to H$ morphism in $\VRSu$, we set $T_u(\varphi)\coloneqq \varphi$.

We denote by $E_u$ the functor 
$$E_u\colon \RSu\to \VRSu$$
that assigns to an object $G\in  \RSu$ the set $G$, endowed with the operations of Riesz space of $G$, and it is enriched with the operation $\bigvee\limits_{n \geq 1  }^g f_n\coloneqq \sup_{n\geq 1  }\{f_n\land g\}$ and with the designated weak  unit of $G$ as the interpretation of the constant symbol $1$. For $\varphi\colon G\to H$ morphism in $\RSu$, we set $E_u(\varphi)\coloneqq \varphi$.

\begin{thm}
	$T_u\colon \VRSu\to \RSu$ and $E_u\colon \RSu\to \VRSu$ are inverse functors.
\end{thm}

\begin{cor}
	The category of Dedekind $\sigma$-complete Riesz spaces with weak  unit is an infinitary variety.
\end{cor}

\section{The varieties  $\VG$, $\VGu$, $\VRS$, $\VRSu$ are generated by $\R$}\label{S. generation}

The set $\R$ has a canonical structure of $\ell$-group (operations: $\{0,+,-,\lor,\land\}$) and of Riesz space (operations: $\{0,+,-,\lor,\land,\}\cup \{\lambda\cdot -\mid\lambda\in\R\}$). Additionally, we may consider the operation $\bigvee\limits_{n\geq 1  }^g f_n=\sup_{n \geq 1  }\{f_n\land g\}$ and the element $1$ (thought as operation of arity $0$).

It takes some easy verifications to see that 
\begin{enumerate}
	\item $\left(\R, \left\{0,+,-,\lor,\land,\bigvee\limits^-\right\}\right)\in\VG$;
	\item $\left(\R, \left\{0,+,-,\lor,\land,\bigvee\limits^-,1\right\}\right)\in\VGu$;
	\item $\left(\R, \left\{0,+,-,\lor,\land,\bigvee\limits^-\right\}\cup\{\lambda\cdot -\mid\lambda\in\R\}\right) \in\VRS$;
	\item $\left(\R, \left\{0,+,-,\lor,\land,\bigvee\limits^-,1\right\}\cup\{\lambda\cdot -\mid\lambda\in\R\}\right)\in\VRSu$.
\end{enumerate}

We will show that 
\begin{enumerate}
	\item the variety $\VG$ is generated by $\left(\R, \left\{0,+,-,\lor,\land,\bigvee\limits^-\right\}\right)$;
	\item the variety $\VGu$ is generated by $\left(\R, \left\{0,+,-,\lor,\land,\bigvee\limits^-,1\right\}\right)$;
	\item the variety $\VRS$ is generated by $\left(\R, \left\{0,+,-,\lor,\land,\bigvee\limits^-\right\}\cup\{\lambda\cdot -\mid\lambda\in\R\}\right)$;
	\item the variety $\VRSu$ is generated by 
	
	$\left(\R, \left\{0,+,-,\lor,\land,\bigvee\limits^-,1\right\}\cup\{\lambda\cdot -\mid\lambda\in\R\}\right)$.
\end{enumerate}
The proof of these results depends on Theorem  \ref{t. Nepalese} below. A proof can be found  in \cite{Nepalese},  and can also be recovered  from the combination of \cite{LoomisRevisited} and \cite{SmallRieszSpaces}. The theorem and its variants have a long history: for a fuller bibliographic account please see \cite{LoomisRevisited}.

\begin{defn}
	Given a set $X$, a \emph{$\sigma$-ideal of subsets of $X$} is a family $\mathcal{F}$ of subsets of $X$ such that
	\begin{enumerate}
		\item $\emptyset\in\mathcal{F}$;
		\item $B\in\mathcal{F}, A\subseteq B\Rightarrow A\in\mathcal{F}$:
		\item $(A_n)_{n\geq 1  }\subseteq \mathcal{F}\Rightarrow \bigcup_{n\geq 1  }A_n\in\mathcal{F}$.
	\end{enumerate}
\end{defn}

Let $X$ be a set and $\mathcal{I}$ a $\sigma$-ideal of subsets of $X$. For $f,g\in\R^X$, we write $f\sim_\mathcal{I}g$ if $\{x\in X\mid f(x)\neq g(x) \}\in\mathcal{I}$. $\sim_\mathcal{I}$ is an equivalence relation on $\R^X$ and we denote by $\frac{\R^X}{\mathcal{I}}$ the quotient of $\R^X$ by $\sim_\mathcal{I}$. The operations of $\ell$-groups, the operation $\bigvee\limits^-,$ and the costant $1$ are well-defined in $\frac{\R^X}{\mathcal{I}}$ in the standard way (for example $\bigvee\limits_{n \geq 1  }^{[g]} [f_n]=\left[\bigvee\limits_{n \geq 1  }^g f_n\right]$)
with the map $[-]\colon \R^X\twoheadrightarrow\frac{\R^X}{\mathcal{I}}$ preserving the designated operation. This works because the operations at issue are finitary or countably infinitary.

\begin{rem}\label{r. surj order}
	Let $(A,\leq)$ and $(B,\leq)$ be partially ordered sets, and $\varphi\colon A\to B$ a surjective order-preserving map. Then $\varphi$ preserves every existing supremum.
\end{rem}

\begin{rem}\label{r. sups are preserved by *}
	For $X$ a set and $\mathcal{I}$ a $\sigma$-ideal of subsets of $X$, the $\ell$-morphism $[-]\colon\R^X\twoheadrightarrow \frac{\R^X}{\mathcal{I}}$ is $\sigma$-continuous. Indeed, $[-]$ is a surjective order-preserving map and therefore, by Remark \ref{r. surj order},  it preserves every existing supremum.
\end{rem}
\begin{prop}\label{p. frac is sigma-complete}
	For $X$ a set and $\mathcal{I}$ a $\sigma$-ideal of subsets of $X$, the $\ell$-group $\frac{\R^X}{\mathcal{I}}$ is Dedekind $\sigma$-complete.
\end{prop}
\begin{proof}
	Let $([f_n])_{n\geq 1  }$ be a countable collection in $\frac{\R^X}{\mathcal{I}}$ bounded from above by $[g]$. The collection $(f_n\land g)_{n\geq 1  }$ is bounded from above by $g$ and therefore it admits supremum. By Remark \ref{r. sups are preserved by *}, also the collection $([f_n\land g])_{n\geq 1  }=([f_n]\land [g])_{n\geq 1  }=([f_n])_{n\geq 1  }$ admits supremum.
\end{proof}

\begin{thm}[Loomis-Sikorski Theorem for Riesz spaces]\label{t. Nepalese}
	Let $G$ be a Dedekind $\sigma$-complete Riesz space. Then there exist a set $X$, a boolean $\sigma$-ideal $\mathcal{I}$ of subsets of $X$ and an injective $\sigma$-continuous Riesz morphism $\varphi\colon G\xhookrightarrow{} \frac{\R^X}{\mathcal{I}}$.
\end{thm}
\begin{proof}
	See  \cite{Nepalese}. 
\end{proof}

\subsection{The variety $\VG$ is generated by $\R$}

\begin{thm}[Loomis-Sikorski Theorem for $\ell$-groups]\label{t. Nepalese for groups}
	Let $G$ be a Dedekind $\sigma$-complete  $\ell$-group. Then there exist a set $X$, a boolean $\sigma$-ideal $\mathcal{I}$  of subsets of $X$ and an injective $\sigma$-continuous $\ell$-morphism $\varphi\colon G\xhookrightarrow{} \frac{\R^X}{\mathcal{I}}$.
\end{thm}
\begin{proof}
	There exists a Dedekind $\sigma$-complete Riesz space $H$ and an injective $\sigma$-continuous $\ell$-morphism $\iota\colon G\xhookrightarrow{} H$; see, e.g., \cite{EmbeddingGroupInRiesz}. Applying Theorem \ref{t. Nepalese} to the Dedekind $\sigma$-complete Riesz space $H$, we obtain an injective $\sigma$-continuous Riesz morphism $\varphi'\colon H\xhookrightarrow{} \frac{\R^X}{\mathcal{I}}$. The composition $\varphi=\varphi'\circ\iota\colon G\xhookrightarrow{} \frac{\R^X}{\mathcal{I}}$ is an injective $\sigma$-continuous $\ell$-morphism, since both $\iota$ and $\varphi'$ are injective $\sigma$-continuous $\ell$-morphisms.
\end{proof}

\begin{rem}\label{r. two ways without u}
	$\frac{\R^X}{\mathcal{I}}$ may be thought of as an object of $\VG$ in two ways:
	\begin{enumerate}
		\item \label{i. first way without u}$\frac{\R^X}{\mathcal{I}}$ inherits from $\R^X$ the structure of $\ell$-group. By Proposition \ref{p. frac is sigma-complete}, with such structure $\frac{\R^X}{\mathcal{I}}$ is Dedekind $\sigma$-complete. Hence $\frac{\R^X}{\mathcal{I}}\in\G$, and $F\left(\frac{\R^X}{\mathcal{I}}\right)\in\VG$.
		\item Since $\R\in\VG$, and $\sim_\mathcal{I}$ is a congruence in the language of $\VG$, $\frac{\R^X}{\mathcal{I}}$ inherits a structure of object of $\VGu$ as a quotient of a power of $\R$. We will denote this object by $\frac{F(\R)^X}{\mathcal{I}}$.
	\end{enumerate}
\end{rem}

\begin{lem}\label{l. quotient of power}
	$F\left(\frac{\R^X}{\mathcal{I}}\right)=\frac{F(\R)^X}{\mathcal{I}}$, i.e.\ $F\left(\frac{\R^X}{\mathcal{I}}\right)$ and $\frac{F(\R)^X}{\mathcal{I}}$ identify the same object of $\VG$.
\end{lem}
\begin{proof}
	The operations of $\ell$-groups are the same in $F\left(\frac{\R^X}{\mathcal{I}}\right)$ and $\frac{F(\R)^X}{\mathcal{I}}$ because in both cases they are defined in the standard way (for example $[f]\lor [g]\coloneqq[f\lor g]$).
	
	Let us see how the operation $\bigvee\limits^-$ is defined in $F\left(\frac{\R^X}{\mathcal{I}}\right)$. Let $([f_n])_{n\geq 1  }\subseteq F\left(\frac{\R^X}{\mathcal{I}}\right),[g]\in F\left(\frac{\R^X}{\mathcal{I}}\right)$.
	$$\bigvee_{n\geq 1  }^{[g] }[f_n]\stackrel{\text{def. of }F}{=}\sup_{n\geq 1  }\{[f_n]\land [g] \}=\sup_{n\geq 1  }\{[f_n\land g] \}\stackrel{\text{Rem. }\ref{r. sups are preserved by *}}{=}\left[\sup_{n\geq 1  }\{f_n\land g \}\right].$$
	We further know $(\sup_{n\geq 1  }\{f_n\land g \}) (x)=\sup_{n\geq 1  }\{f_n(x)\land g(x) \}$, since the order in the $\ell$-group $\R^X$ is defined by $f\leq g\Leftrightarrow f\lor g=g$, and $(f\lor g)(x)\coloneqq f(x)\lor g(x)$.
	
	In $\frac{F(\R)^X}{\mathcal{I}}$, $\bigvee\limits^-$ is defined as follows.
	$$\bigvee_{n\geq 1  }^{[g] }[f_n]=\left[\bigvee_{n\geq 1  }^{g }f_n\right].$$
	We know $\left(\bigvee\limits_{n\geq 1  }^{g }f_n\right)(x)=\bigvee\limits_{n\geq 1  }^{g(x)}f_n(x)=\sup_{n\geq 1  }\{f_n(x)\land g(x) \}$.
	
	We have shown that the operation $\bigvee\limits^-$ is the same in $F\left(\frac{\R^X}{\mathcal{I}}\right)$ and $\frac{F(\R)^X}{\mathcal{I}}$.
\end{proof}

\begin{thm}\label{t. generation v}
	The variety $\VG$ is generated by $\R$.
\end{thm}
\begin{proof}
	Let $G\in \VG$. $U(G)\in\G$. By Theorem \ref{t. Nepalese for groups}  we have an injective $\sigma$-continuous $\ell$-morphism  $\varphi\colon U(G)\xhookrightarrow{} \frac{\R^X}{\mathcal{I}}$. $U(G)\in\G, \frac{\R^X}{\mathcal{I}}\in\G$ (as shown in Item \eqref{i. first way without u} in Remark \ref{r. two ways without u}) , and $\varphi$ is a morphism in $\G$. We can therefore apply the functor $F$ to $\varphi$, obtaining an injective morphism in $\VG$
	$$\underbrace{F(\varphi)}_{=\varphi}\colon \underbrace{F(U(G))}_{=G}\xhookrightarrow{} F\left(\frac{\R^X}{\mathcal{I}}\right)$$
	We see that $F(\varphi)$ injects $G$ in $F\left(\frac{\R^X}{\mathcal{I}}\right)$, which, by Lemma \ref{l. quotient of power}, is a quotient of a power of $\R$ in $\VG$.
\end{proof}

\subsection{The variety $\VGu$ is generated by $\R$}\label{S. generation groups u}

\begin{lem}\label{l. can be chosen positive}
	Let $\varphi\colon G\twoheadrightarrow H$ be a surjective $\ell$-morphism. For $h\in H^+$, there exists an element $g\in G^+$ such that $h=\varphi(g)$.
\end{lem}
\begin{proof}
	Let $f$ be such that $\varphi(f)=h$. Then $\varphi(f^+)=\varphi(f)^+=h^+=h$. Thus $g\coloneqq f^+$ satisfies the desired properties.
\end{proof}

\begin{thm}[Loomis-Sikorski Theorem for $\ell$-groups with weak unit]\label{t. Nepalese for groups unit}
	Let $G$ be a Dedekind $\sigma$-complete  $\ell$-group with weak  unit. Then there exist a set $X$, a boolean $\sigma$-ideal $\mathcal{I}$ of subsets of $X$ and an injective $\sigma$-continuous $\ell$-morphism $\varphi\colon G\xhookrightarrow{} \frac{\R^X}{\mathcal{I}}$ such that $\varphi(1_G)=[1_{\R^X}]$.
\end{thm}
\begin{proof}
	By Theorem \ref{t. Nepalese for groups}, we have an injective $\sigma$-continuous $\ell$-morphism $\psi\colon G\xhookrightarrow{} \frac{\R^Y}{\mathcal{J}}$.
	
	Since $1_G\geq 0$, we have $\psi(1_G)\geq 0$. By Lemma \ref{l. can be chosen positive}, we can choose an element $u\in \left(\R^Y\right)^+$ such that $[u]_\mathcal{J}=\psi(1_G)$. Define $X\coloneqq\{x\in X\mid  u(x)> 0\}$. Define $\mathcal{I}\coloneqq\{J\cap X \mid J\in\mathcal{J} \}=\{J\subseteq X \mid J\in\mathcal{J} \}$. Note that $\mathcal{I}$ is a $\sigma$-ideal of subsets of $X$. Define the map $-_{\mid X}\colon\R^Y\twoheadrightarrow \R^X$ by $f\mapsto f_{\vert X}$. Note that $-_{\mid X}$ is a surjective $\ell$-morphism.
	\[ 
	\begin{tikzpicture}[node distance=1.7 cm, auto]
	\node (11) {$\R^Y$};
	\node (12) [right of=11]{$\R^X$};
	\node (21) [below of=11]{$\frac{\R^Y}{\mathcal{J}}$};
	\node (22) [right of=21]{$\frac{\R^X}{\mathcal{I}}$};
	
	{\draw[->>] (11) to node {$-_{\mid X}$}  (12);}
	{\draw[->>] (11) to node [swap]{$[-]_\mathcal{J}$}  (21);}
	{\draw[->>] (12) to node {$[-]_\mathcal{I}$}  (22);}
	\end{tikzpicture}
	\]
	We have	$ker\left([-]_\mathcal{J}\right)\subseteq ker\left([-]_\mathcal{I}\circ r\right)$. Indeed, let $f\in ker\left([-]_\mathcal{J}\right)$. Then $\{y\in Y\mid  f(y)\neq0\}\in \mathcal{J}$. Then $\{x\in X\mid  f_{\vert X}(x)\neq0\}=\{y\in Y\mid  f(y)\neq0\}\cap X\in \{J\cap X \mid J\in\mathcal{J} \}=\mathcal{I}$. Hence $[f_{\vert X}]_\mathcal{I}=0$, i.e.\ $f\in ker\left([-]_\mathcal{I}\circ -_{\mid X}\right)$.
	
	Since $ker\left([-]_\mathcal{J}\right)\subseteq ker\left([-]_\mathcal{I}\circ -_{\mid X}\right)$, by the universal property of the quotient there exists a unique $\ell$-morphism $\rho\colon \frac{\R^Y}{\mathcal{J}}\to \frac{\R^X}{\mathcal{I}}$ such that the following diagram commutes.
	\[ 
	\begin{tikzpicture}[node distance=1.7 cm, auto]
	\node (11) {$\R^Y$};
	\node (12) [right of=11]{$\R^X$};
	\node (21) [below of=11]{$\frac{\R^Y}{\mathcal{J}}$};
	\node (22) [right of=21]{$\frac{\R^X}{\mathcal{I}}$};
	
	{\draw[->>] (11) to node {$-_{\mid X}$}  (12);}
	{\draw[->>] (11) to node [swap]{$[-]_\mathcal{J}$}  (21);}
	{\draw[->>] (12) to node {$[-]_\mathcal{I}$}  (22);}
	{\draw[->, dashed] (21) to node {$\rho$}  (22);}
	\end{tikzpicture}
	\]
	
	Being $-_{\mid X} $ and $[-]_{\mathcal{I}}$ surjective, $\rho$ must be surjective. By Remark \ref{r. surj order},	$\rho$ preserves the existing countable suprema.
	
	Define the map $m\colon \R^X\xrightarrow{\sim} \R^X$ by $(m(f))(x)=\frac{1}{u(x)}f(x)$. Note that, by definition of $X$, $u(x)>0$ for every $x\in X$, and therefore the map is well defined.	For every $\lambda\in\R$, $\lambda>0$, the map $m_{\lambda}\colon \R\xrightarrow{\sim} \R$ defined by $g\mapsto\lambda g$ is an $\ell$-isomorphism. This suggests that $m$ is an $\ell$-isomorphism. In fact, its inverse is the $\ell$-morphism $m^{-1}\colon \R^X\xrightarrow{\sim} \R^X$ defined by $(m^{-1}(g))(x)=u(x)g(x)$.
	\[ 
	\begin{tikzpicture}[node distance=1.7 cm, auto]
	\node (11) {$\R^X$};
	\node (12) [right of=11]{$\R^X$};
	\node (21) [below of=11]{$\frac{\R^X}{\mathcal{I}}$};
	\node (22) [right of=21]{$\frac{\R^X}{\mathcal{I}}$};
	
	{\draw[right hook->>] (11) to node {$m$}  (12);}
	{\draw[->>] (11) to node [swap]{$[-]_\mathcal{I}$}  (21);}
	{\draw[->>] (12) to node {$[-]_\mathcal{I}$}  (22);}
	\end{tikzpicture}
	\]
	
	Since $[-]_\mathcal{J}$ and $[-]_\mathcal{I}\circ m$ are surjective $\ell$-morphism with same kernel, using the universal property of quotients we obtain that there exists an $\ell$-isomorphism $\eta\colon\frac{\R^X}{\mathcal{I}}\xrightarrow{\sim}\frac{\R^X}{\mathcal{I}}$ which makes the following diagram commute.
	\[ 
	\begin{tikzpicture}[node distance=1.7 cm, auto]
	\node (11) {$\R^X$};
	\node (12) [right of=11]{$\R^X$};
	\node (21) [below of=11]{$\frac{\R^X}{\mathcal{I}}$};
	\node (22) [right of=21]{$\frac{\R^X}{\mathcal{I}}$};
	
	{\draw[right hook->>] (11) to node {$m$}  (12);}
	{\draw[->>] (11) to node [swap]{$[-]_\mathcal{I}$}  (21);}
	{\draw[->>] (12) to node {$[-]_\mathcal{I}$}  (22);}
	{\draw[right hook->>, dashed] (21) to node {$\eta$}  (22);}
	\end{tikzpicture}
	\]
	
	We have the following diagram.
	\[ 
	\begin{tikzpicture}[node distance=1.7 cm, auto]
	\node (10) {};
	\node (20) [below of=10] {$G$};
	\node (11) [right of=10] {$\R^Y$};
	\node (12) [right of=11]{$\R^X$};
	\node (21) [below of=11]{$\frac{\R^Y}{\mathcal{J}}$};
	\node (22) [right of=21]{$\frac{\R^X}{\mathcal{I}}$};
	\node (13) [right of=12]{$\R^X$};
	\node (23) [right of=22]{$\frac{\R^X}{\mathcal{I}}$};
	
	{\draw[right hook->] (20) to node {$\psi$}  (21);}
	{\draw[->>] (11) to node {$-_{\mid X} $}  (12);}
	{\draw[->>] (11) to node [swap]{$[-]_\mathcal{J}$}  (21);}
	{\draw[->>] (12) to node {$[-]_\mathcal{I}$}  (22);}
	{\draw[->>] (21) to node {$\rho$}  (22);}
	
	{\draw[right hook->>] (12) to node {$m$}  (13);}
	{\draw[->>] (13) to node {$[-]_\mathcal{I}$}  (23);}
	{\draw[right hook->>] (22) to node {$\eta$}  (23);}
	\end{tikzpicture}
	\]
	
	Define $\varphi\coloneqq\eta\circ\rho\circ\psi\colon G\to\frac{\R^X}{\mathcal{I}}$. Then $\varphi$ has the desired properties. Indeed,
		\begin{enumerate}
			\item $\varphi$ is a $\sigma$-continuous $\ell$-morphism since $\eta$, $\rho$ and $\psi$ are such.
			\item We show that $\varphi$ is injective. In fact, being $\eta$ an $\ell$-isomorphism, it is enough to show that $\rho\circ\psi$ is injective. Let $g\in G$ be such that $(\rho\circ \psi)(g)=0$. Let $f\in\R^Y$ such that $[f]_\mathcal{J}=\psi(g)$. Then $[f_{\mid X}]_\mathcal{I}=\rho([f]_\mathcal{J})=\rho(\psi(g))=0$. This means $I\coloneqq\{x\in X\mid  f(x)\neq0\}=\{x\in X\mid  f_{\mid X}(x)\neq0\}\in\mathcal{I}$.
			
			$\{y\in Y\mid (u\land \lvert f\rvert)(y)\neq 0 \}=\{y\in Y\mid u(y)\neq 0\}\cap\{y\in Y\mid f(y)\neq 0 \}=X\cap\{y\in Y\mid f(y)\neq 0\}=\{x\in X\mid f(x)\neq 0\}=I\in \mathcal{I}\subseteq \mathcal{J}$. Hence $0=[u\land \lvert f \rvert]_\mathcal{J}=[u]_\mathcal{J}\land \lvert [f]_\mathcal{J}\rvert=\psi(1_G)\land \lvert \psi(g)\rvert=\psi(1_G\land \lvert g\rvert)$.
			By injectivity of $\psi$, it follows that $1_G\land \lvert g\rvert=0$, hence $\vert g\rvert =0$, hence $g=0$.

			\item $\varphi(1_G)=[1_{\R^X}]_\mathcal{I}$. Indeed 
			$$\varphi(1_G)=\eta\rho\psi(1_G)=\eta\rho([u]_\mathcal{J})=[m(r(u))]_\mathcal{I}=[m(u_{\vert X})]_\mathcal{I}$$
			$$(m(u_{\vert X}))(x)\stackrel{\text{def. of }m}{=}\frac{1}{u(x)}u(x)=1$$
			Hence $m(u_{\vert X})=1_{\R^X}$ and $\varphi(1_G)=[1_{\R^X}]_\mathcal{I}$.
		\end{enumerate}	
\end{proof}

\begin{lem}\label{l. 1 is weak unit}
	$[1_{\R^X}]$ is a weak  unit for $\frac{\R^X}{\mathcal{I}}$. 
\end{lem}
\begin{proof}
	First, $[1]\geq[0]$ since $[-]$ preserves the order. Second, let $[1]\land [f]=0$. Then $[1\land f]=0$.  Thus $\{x\in X\mid f(x)\neq 0 \}=\{x\in X\mid (1\land f)(x)\neq 0\}\in I$. Thus $[f]=0$.
\end{proof}

\begin{rem}\label{r. two ways}
	$\frac{\R^X}{\mathcal{I}}$ may be thought of as an object of $\VGu$ in two ways:
	\begin{enumerate}
		\item \label{i. first way} $\frac{\R^X}{\mathcal{I}}$ inherits from $\R^X$ the structure of $\ell$-group. By Proposition \ref{p. frac is sigma-complete}, with such structure $\frac{\R^X}{\mathcal{I}}$ is Dedekind $\sigma$-complete. By Lemma \ref{l. 1 is weak unit}, $[1_{\R^X}]$ is a weak unit for $\frac{\R^X}{\mathcal{I}}$. Thus $\frac{\R^X}{\mathcal{I}}\in\Gu$, and $F_u\left(\frac{\R^X}{\mathcal{I}}\right)\in\VGu$.
		\item Since $\R\in\VGu$, and $\sim_\mathcal{I}$ is a congruence in the language of $\VGu$, $\frac{\R^X}{\mathcal{I}}$ inherits a structure of object of $\VGu$ as quotient of a power of $\R$. We will denote this object as $\frac{F_u(\R)^X}{\mathcal{I}}$.
	\end{enumerate}
\end{rem}

\begin{lem}\label{l. quotient of power u}
	$F_u\left(\frac{\R^X}{\mathcal{I}}\right)=\frac{F_u(\R)^X}{\mathcal{I}}$, i.e.\ $F_u\left(\frac{\R^X}{\mathcal{I}}\right)$ and $\frac{F_u(\R)^X}{\mathcal{I}}$ identify the same object of $\VGu$.
\end{lem}
\begin{proof}
	As far as it concerns the operations of $\ell$-groups and the operation $\bigvee\limits^-$, they are defined in the same ways in the two structures, as already shown in Lemma \ref{l. quotient of power}. Moreover, in $F_u\left(\frac{\R^X}{\mathcal{I}}\right)$, the constant $1$ is defined as $1_{\frac{\R^X}{\mathcal{I}}}\coloneqq [1_{\R^X}]$, which coincides with the interpretation of the constant symbol $1$ in $\frac{F_u(\R)^X}{\mathcal{I}}$.
\end{proof}

\begin{thm}\label{t. generation vu}
	The variety $\VGu$ is generated by $\R$.
\end{thm}
\begin{proof}
	Let $G\in \VGu$. $U_u(G)\in\Gu$. By Theorem \ref{t. Nepalese for groups unit}  we have an injective $\sigma$-continuous $\ell$-morphism  $\varphi\colon U_u(G)\xhookrightarrow{} \frac{\R^X}{\mathcal{I}}$ such that $\varphi(1_G)=[1_{\R^X}]$. $U_u(G)\in\Gu, \frac{\R^X}{\mathcal{I}}\in\Gu$ (as shown in Item \eqref{i. first way} in Remark \ref{r. two ways}) and $\varphi$ is a morphism in $\Gu$. We can therefore apply the functor $F_u$ to $\varphi$, and we obtain an injective morphism in $\VGu$
	$$\underbrace{F_u(\varphi)}_{=\varphi}\colon \underbrace{F_u(U_u(G))}_{=G}\xhookrightarrow{} F_u\left(\frac{\R^X}{\mathcal{I}}\right)$$
	$F_u(\varphi)$ injects $G$ in $F_u\left(\frac{\R^X}{\mathcal{I}}\right)$, which, by Lemma \ref{l. quotient of power u}, is a quotient of a power of $\R$.
\end{proof}

\subsection{The varieties $\VRS$ and $\VRSu$ are generated by $\R$}

The proofs of the following two theorems are analogous to the proofs of Theorem \ref{t. generation v} and Theorem \ref{t. generation vu}.

\begin{thm}\label{t. generation r}
	The variety $\VRS$ is generated by $\R$.
\end{thm}
\begin{thm}
	The variety $\VRSu$ is generated by $\R$.
\end{thm}
\subsection{Conclusion}
Summing Section \ref{S. generation} up, we have the following results.
\begin{thm}\label{t. HSP}
\begin{enumerate}
	\item The variety $\VG$ is generated by $\left(\R, \left\{0,+,-,\lor,\land,\bigvee\limits^-\right\}\right)$;
	\item The variety $\VGu$ is generated by $\left(\R, \left\{0,+,-,\lor,\land,\bigvee\limits^-,1\right\}\right)$;
	\item The variety $\VRS$ is generated by $\left(\R, \left\{0,+,-,\lor,\land,\bigvee\limits^-\right\}\cup\{\lambda\cdot -\mid\lambda\in\R\}\right)$;
	\item The variety $\VRSu$ is generated by 
	
	$\left(\R, \left\{0,+,-,\lor,\land,\bigvee\limits^-,1\right\}\cup\{\lambda\cdot -\mid\lambda\in\R\}\right)$.
\end{enumerate}
\end{thm}

\section{The quasi-varieties $\VG$, $\VGu$, $\VRS$, $\VRSu$ are generated by $\R$}

\begin{defn}
	Given a similarity type $\mathcal{L}$, a \emph{quasi-equation with countably many premises} (in $\mathcal{L}$) is an expression of the following form.	
		\begin{align*}
	\begin{cases}
	\tau_1=\rho_1&\\
	\tau_2=\rho_2&\\
	\vdots &\\
	\end{cases}&\Rightarrow\tau=\rho
	\end{align*}	
	where $I$ is a set and  $\tau,\rho,\tau_1,\rho_1,\tau_2,\rho_2,\dots$ are terms in $\mathcal{L}$. We say that such a quasi-equation \emph{holds} in an $\mathcal{L}$-algebra $A$ if the conclusion holds in $A$ whenever each premise hold in $A$.
\end{defn}

\begin{example}
	One example of quasi-equation with countably many premises is the archimedean property.
	\begin{quote}
		If $a\geq 0$ and $na\leq b$ for every $n\geq 1$, then $a=0$.
	\end{quote}
\end{example}

\begin{lem}\label{l. property in R}
	For $a,b,a_1,b_1,a_2,b_2\dots\in\R$, if
	$$[a_1=b_1, a_2=b_2,\dots] \Rightarrow a=b$$
	then
	$$\bigvee\limits_{n,k \geq 1  }^{|a-b|}k|a_n-b_n|=|a-b|.$$
\end{lem}
\begin{proof}
	Let us suppose $[a_1=b_1, a_2=b_2, \dots] \Rightarrow a=b$. This means $a=b$ or there exists $m\geq 1 $ such that $a_m\neq b_m$. If $a=b$, then $\bigvee\limits_{n,k \geq 1  }^{|a-b|}k|a_n-b_n|=\bigvee\limits_{n,k \geq 1  }^{0}k|a_n-b_n|=0=|a-b|$. If, instead, $a_m\neq b_m$ for some $m\geq 1  $, then $|a-b|\geq \bigvee\limits_{n,k \geq 1  }^{|a-b|}k|a_n-b_n|\geq\bigvee\limits_{k \geq 1  }^{|a-b|}k|a_m-b_m|\stackrel{|a_m-b_m|>0}{=}|a-b|$.
\end{proof}

\begin{thm}\label{t. finally mf gen}
	Let $\mathcal{Z}$  be the variety generated by $(\R,\mathcal{L})$, where $0,+,-,\lor,\land,\bigvee\limits^-\in\mathcal{L}$. If a quasi-equation with countably many premises in the language $\mathcal{L}$ holds in $\R$, then it holds in every $G\in\mathcal{Z}$.
\end{thm}
\begin{proof}
	Let us suppose that the quasi-equation
	$$[\tau_1=\rho_1,\tau_2=\rho_2,\dots] \Rightarrow \tau=\rho$$
	holds in $\R$. By Lemma \ref{l. property in R}, the equation
	$$\bigvee\limits_{n,k \geq 1  }^{|\tau-\rho|}k|\tau_n-\rho_n|=|\tau-\rho|$$	
	holds in $\R$, and therefore in every $G\in\mathcal{Z}$. We now see that in such a $G$, if $\tau_n=\rho_n$ for all $n$, then
	$$|\tau-\rho|=\bigvee\limits_{n,k \geq 1  }^{|\tau-\rho|}k\lvert\tau_n-\rho_n\rvert=\bigvee\limits_{n,k \geq 1  }^{|\tau-\rho|}0=0.$$	
	In conclusion, $\tau=\rho$.
\end{proof}

\begin{cor}\label{c. generation as quasi-variety}
	If a quasi-equation with countably many premises in the language of $\VG$ (resp. $\VGu$, $	\VRS$, $\VRSu$) holds in $\R$, then it holds in every object of $\VG$ (resp. $\VGu$, $	\VRS$, $\VRSu$).
\end{cor}
\begin{proof}
	By Theorem \ref{t. HSP}, the hypotheses of Theorem \ref{t. finally mf gen} are satisfied by $\VG$, $\VGu$, $\VRS$ and $\VRSu$, and the thesis follows.
\end{proof}

\begin{example}
	We illustrate Corollary \ref{c. generation as quasi-variety} with an example. Since the archimedean property is a quasi-equation with countably many premises, and since in $\R$ such property is easily seen to hold, by Corollary \ref{c. generation as quasi-variety} the well-known fact that every Dedekind $\sigma$-complete $\ell$-group is archimedean immediately follows.
\end{example}

\begin{example} We illustrate Corollary \ref{c. generation as quasi-variety} with another example. Corollary \ref{c. generation as quasi-variety} ensures that the following statement, which is a particular case of the general distributivity law in Lemma \ref{l. distriutivity a la BKW}, can be proved just by checking its validity for $G=\R$.
	\begin{quote}
		Let $G$ be a Dedekind $\sigma$-complete  $\ell$-group, and let $(x_n)_{n\geq 1  }\subseteq G$. If $\sup_{n\geq 1  } x_n$ exists, then, for every $a\in G$, $\sup_{n\geq 1  }\{a\land x_n\}$ exists and 
		$$a\land \left(\sup_{n\geq 1  } x_n\right)=\sup_{n\geq 1  } \{a\land x_n\}.$$
	\end{quote}
	Indeed, the statement is equivalent to 
	$$b=\sup_{n\geq 1  }x_n\Rightarrow a\land b=\sup_{n\geq 1  }\{a\land x_n \}$$
	i.e.\ to a conjunction of quasi-equations with countably many premises, as the next proposition shows.   
	
	\begin{prop}
		Let $G\in\VG$, $(f_n)_{n\geq 1  }\subseteq G$ and $g\in G$. Then
		$$g=\sup_{n\geq 1  } f_n\Leftrightarrow \begin{cases}
		g=\bigvee\limits_{n\geq 1  }^g f_n&\\
		f_1\land g=f_1&\\
		f_2\land g=f_2&\\
		\vdots &
		\end{cases} $$
	\end{prop}
	\begin{proof}
		If $g=\sup_{n\geq 1  } f_n$, then $f_n\land g=f_n$ for every $n\geq 1  $. Moreover, by Proposition \ref{supinf}, $\bigvee\limits_{n\geq 1  }^g f_n=\sup_{n\geq 1  }\{f_n\land g\}=\sup_{n\geq 1  }f_n=g$.
		
		Suppose now $f_n\land g=f_n$ for every $n\geq 1  $ and $g=\bigvee\limits_{n\geq 1  }^g f_n$. Then, by Proposition \ref{supinf},
		$g=\bigvee\limits_{n\geq 1  }^g f_n=\sup_{n\geq 1  }\{f_n\land g\}=\sup_{n\geq 1  }f_n$.
	\end{proof}
\end{example}


\begin{thebibliography}{10}

\bibitem{Aliprantis_Border2006}
C.~D. Aliprantis and K.~C. Border.
\newblock {\em Infinite dimensional analysis}.
\newblock Springer, Berlin, third edition, 2006.
\newblock A hitchhiker's guide.

\bibitem{Aliprantis_Burkinshaw}
C.~D. Aliprantis and O.~Burkinshaw.
\newblock {\em Positive operators}.
\newblock Springer, Dordrecht, 2006.
\newblock Reprint of the 1985 original.

\bibitem{BKW}
A.~Bigard, K.~Keimel, and S.~Wolfenstein.
\newblock {\em Groupes et anneaux r\'eticul\'es}.
\newblock Lecture Notes in Mathematics, Vol. 608. Springer-Verlag, Berlin-New
  York, 1977.

\bibitem{Birkhoff}
G.~Birkhoff.
\newblock {\em Lattice theory}.
\newblock Third edition. American Mathematical Society Colloquium Publications,
  Vol. XXV. American Mathematical Society, Providence, R.I., 1967.

\bibitem{Burris_Sanka}
S.~Burris and H.~P. Sankappanavar.
\newblock {\em A course in universal algebra}, volume~78 of {\em Graduate Texts
  in Mathematics}.
\newblock Springer-Verlag, New York-Berlin, 1981.

\bibitem{LoomisRevisited}
G.~Buskes, B.~de~Pagter, and A.~van Rooij.
\newblock The {L}oomis-{S}ikorski theorem revisited.
\newblock {\em Algebra Universalis}, 58(4):413--426, 2008.

\bibitem{SmallRieszSpaces}
G.~Buskes and A.~van Rooij.
\newblock Small {R}iesz spaces.
\newblock {\em Math. Proc. Cambridge Philos. Soc.}, 105(3):523--536, 1989.

\bibitem{Nepalese}
G.~Buskes and A.~Van~Rooij.
\newblock Representation of {R}iesz spaces without the {A}xiom of {C}hoice.
\newblock {\em Nepali Math. Sci. Rep.}, 16(1-2):19--22, 1997.

\bibitem{EmbeddingGroupInRiesz}
M.~A. Lepellere and A.~Valente.
\newblock Embedding of {A}rchimedean {$l$}-groups in {R}iesz spaces.
\newblock {\em Atti Sem. Mat. Fis. Univ. Modena}, 46(1):249--254, 1998.

\bibitem{LuxZaan}
W.~A.~J. Luxemburg and A.~C. Zaanen.
\newblock {\em Riesz spaces. {V}ol. {I}}.
\newblock North-Holland Publishing Co., Amsterdam-London; American Elsevier
  Publishing Co., New York, 1971.
\newblock North-Holland Mathematical Library.

\end{thebibliography}

\end{document}